\documentclass[12pt]{amsart}
\usepackage{amsmath,amsthm,amsfonts,amssymb,times,latexsym,mathabx,url}

\newtheorem{theorem}{Theorem}[section]

\newtheorem{lem}[theorem]{Lemma}

\numberwithin{equation}{section}

\makeatletter
\renewcommand{\pmod}[1]{\allowbreak\mkern7mu({\operator@font mod}\,\,#1)}
\makeatother

\renewcommand{\a}{\alpha}
\renewcommand{\b}{\beta}
\newcommand{\g}{\gamma}
\renewcommand{\d}{\delta}
\newcommand{\e}{\varepsilon}
\renewcommand{\l}{\lambda}
\renewcommand{\o}{\omega}
\renewcommand{\O}{\Omega}

\newcommand{\C}{\mathcal{C}}

\newcommand{\E}{\mathbb{E}}

\newcommand{\N}{\mathbb{N}}
\newcommand{\G}{\Gamma}
\renewcommand{\P}{\mathbb{P}}

\newcommand{\R}{{\mathbb R}}
\newcommand{\s}{\sigma}

\newcommand{\Z}{\mathbb{Z}}

\newcommand{\ro}{\Rightarrow}

\newcommand{\1}{\mathbf{1}}

\newcommand{\bo}{\boldsymbol{\omega}}
\newcommand{\bZ}{\mathbf Z}

\newcommand{\h}{\mathfrak{h}}

\renewcommand{\leq}{\leqslant}
\renewcommand{\geq}{\geqslant}
\newcommand{\Var}{\operatorname{Var}}

\newcommand{\Pois}{\operatorname{Pois}}
\newcommand{\res}{\operatorname{res}}

\begin{document}

\title{Distinct exponents in the prime factorization}%: \\
%	Erd\H{o}s Problem \#912 and around}
\author{Mikhail R. Gabdullin, Vitalii V. Iudelevich}
 
\address{Department of mathematics, 1409 West Green Street, University of Illinois at Urbana-Champaign, Urbana, IL 61801, USA} 
\email{mikhail.gabdullin@gmail.com}

\address{Faculty of Computer Science, National Research University
``Higher School of Economics'', Pokrovsky blvd.\ 11, Moscow 109028, Russia}
\email{vitaliiiudelevich@gmail.com}

\begin{abstract}
Following Erd\H{o}s (1982) and Sanna (2019), we study the arithmetic function $h(n)$, which is defined to be the number of distinct exponents in the prime factorization of a positive integer $n$. Among other things, we show that
$$
\sum_{n\leq x}h(\phi(n)) \asymp x\left(\frac{\log\log x}{\log\log\log x}\right)^{1/2},
$$
where $\phi$ is the Euler totient function. The key ingredient is the Poisson random model for $\o(n,T)$, the number of the prime divisors of $n$ in a given subset of primes $T$, which was introduced in a recent work of Ford. 

%More detailed abstract, with Tao?
\end{abstract}		

\date{\today}	
\maketitle

\section{Introduction}

For a positive integer $n$, let $h(n)$ be the number of distinct exponents in its factorization (we set $h(1)=0$). This function was previously studied by Erd\H{o}s \cite{Erd82} and Sanna \cite{Sanna19}. We start with a discussion of basic statistics for $h$. It is easy to see that the density of $n\leq x$ with $h(n)\geq k$ goes to zero as $k\to\infty$. In \cite{Sanna19} more precise statements were proved, in particular, explicit expressions for the above densities were given and also it was shown that 
$$
\sum_{n\leq x}f(h(n))=C_fx+O_{\e}(x^{1/2+\e})
$$ 
for each function $f$ which does not grow too fast and each fixed $\e>0$.  %$h(n)$ is essentially determined by the powerful part of $n$

For a positive integer $n$, one can write $n=n'n''$, where $n'$ is square-free and $n''$ is square-full, that is, every prime divisor of $n'$ occurs to the first power in $n$, while every prime divisor of $n''$ occurs to at least the second power. Since
$$
h(n)=\mathbf{1}(n'>1)+h(n''),
$$
the behavior of $h(n)$ is essentially determined by the square-full part of $n$. Let $S$ denote the set of square-full numbers. It is known that
\begin{equation}\label{1.1}
\#\{n\leq x: n\in S\}= \frac{\zeta(3/2)}{\zeta(3)}x^{1/2}+\frac{\zeta(2/3)}{\zeta(2)}x^{1/3}+O(x^{1/5});	
\end{equation}
see \cite{B} and also \cite{S} for a better error term. This connection naturally motivates the search for a similar asymptotic formula for the average value of $h(n)$. By elementary considerations, one can show that
\begin{equation}\label{1.2}
\sum_{n\leq x}h(n)=c_{\h}x+O\left(x^{1/2}(\log x)^2\log\log x\right),		
\end{equation}
where the constant $c_{\h}$ is given by  
\begin{equation}\label{1.3}
c_{\h}=\frac{1}{\zeta(2)}\sum_{b\in S}\frac{1+h(b)}{\psi(b)}	
\end{equation}
and $\psi$ is the multiplicative function $\psi(b)=b\prod_{p|b}(1+1/p)$. Our first result is a sharpened version (\ref{1.2}), in which we identify the second term in the asymptotic expansion.

\begin{theorem}\label{th1.1} 
We have
\begin{equation*}
\sum_{n\leq x}h(n)=c_{\mathfrak{h}}x-\frac{\zeta(3/2)}{\zeta(3)}x^{1/2}+O\Big(x^{1/2}\exp(-c_1(\log x/\log\log x)^{1/3})\Big),		
\end{equation*}
for some absolute $c_1>0$. 
\end{theorem}

It is interesting to note that the second term in Theorem \ref{th1.1} coincides with the main term in \eqref{1.1}, up to the sign. %We do not have a satisfactory explanation for this phenomenon.

To prove Theorem \ref{th1.1}, we use the method of complex integration. The two terms above correspond to the residues of the function $G(s)$, defined in (\ref{2.1}) below, at the points $s=1$ and $s=1/2$, respectively. There is also a potential third term of order $x^{1/3}$ arising from the residue at $s=1/3$ (see Lemma \ref{lem2.2} (iii)). However, this term is smaller than the error term in our estimate, and we were unable to obtain a sufficiently strong error term to detect it, even under the assumption of the Riemann Hypothesis.

Since $h$ is essentially bounded on typical numbers, it is natural to study its behavior on numbers divisible by many large prime powers. One of the first candidates which comes to mind is factorials.  Erd\H{o}s \cite{Erd82} showed that
\begin{equation} \label{1.4}
h(n!) \asymp \left(\frac{n}{\log n}\right)^{1/2} 	
\end{equation}
and conjectured that $h(n!)$ has an asymptotic as $n\to\infty$, writing (\cite{Erd82}, p.27): ``there is no doubt that there is a constant $c>0$ for which 
\begin{equation} \label{1.5}
h(n!)=(c+o(1))\left(\frac{n}{\log n}\right)^{1/2}
\end{equation}
as $n\to\infty$; the proof seems to present very serious difficulties since we do not know enough about the difference of consecutive primes''.  This is listed as Erd\H{o}s Problem $\#912$ on Bloom's website \cite{ErdProbl}. It is natural to try to predict the value of the constant $c$. Let $n$ be large and write $n!=\prod_{p\leq n}p^{\lfloor \frac{n}{p}\rfloor+\lfloor\frac{n}{p^2}\rfloor+\ldots}$. It is not hard to see that the terms $\lfloor\frac{n}{p}\rfloor$ dominate, and, more precisely (see the relation (\ref{3.0}) below),
\begin{equation}\label{1.6}
h(n!)=\#\left\{\left\lfloor\frac{n}{p}\right\rfloor: p\leq n\right\}+O(n^{1/3}).	
\end{equation}
Note that $\lfloor \frac{n}{p}\rfloor=k$ iff $p\in I_k$, where
$$
I_k=\left(\frac{n}{k+1},\,\frac{n}{k}\right].
$$
Thus, by (\ref{1.6})
\begin{equation}\label{1.7}
h(n!)=\sum_{k\geq1}\1\Big(\mbox{there is a prime $p\in I_k$}\Big)+O(n^{1/3})	.
\end{equation}
As we will see below, the main contribution to $h(n!)$ comes from primes of order $(n\log n)^{1/2}$, which means that the main contribution to the right side above comes from $k$ of size roughly $ (n/\log n)^{1/2}$; the corresponding intervals $I_k$ have length about $\log n$. In order to predict how many of $I_k$ contain a prime, Tao used the Cram\'er model for the prime numbers (see comments at \cite{ErdProbl}). We provide the details here, following the standard Cramér model; see also \cite{Cr} or \cite[Chapter 17]{D} for further background. Let $\C$ be the random subset of positive integers greater than or equal to $3$ defined by
\begin{equation}\label{1.8}
\P(l\in\C)=\frac{1}{\log l}
\end{equation}
for $l\geq3$, with the events $l\in \C$ being mutually independent. The set of all primes is believed to behave somewhat similarly to the random set $\C$. Replacing the genuine primes with the set $\C$, from (\ref{1.7}) one can expect $h(n!)$ to be close to the random variable 
\begin{equation}\label{1.9}
h_{\C}(n!)=\sum_{k\geq 1}A_k,
\end{equation}
where $A_k$ is the indicator function of the event $I_k\cap \C \neq  \varnothing$. Tao showed that
\begin{equation*}
\E h_{\C}(n!)=\left(\sqrt{2\pi}+o(1)\right)\left(\frac{n}{\log n}\right)^{1/2}, \quad n\to\infty,	
\end{equation*}
thus suggesting that the constant $c$ in (\ref{1.5}) should be equal to $\sqrt{2\pi}=2.50...$. However, numerical calculations reveal that the ratio $h(n!)/(n/\log n)^{1/2}$ at $n=10^{15}$ is equal to $2.57\ldots$, which is somewhat larger. One may wonder whether this is due to the positivity of the next term in the asymptotic expansion of $h(n!)$. Our next result is the following refined estimate. Here $\g=0.577\ldots$ is the Euler constant.

\begin{theorem}\label{th1.2}
With probability $1$, we have 
$$
h_{\C}(n!)=\frac{\sqrt{2\pi}n^{1/2}}{(\log n)^{1/2}}\left(1-\frac{\log\log n}{2\log n}+\frac{1+\g+\log8}{2\log n}+O\left(\frac{\log\log n}{\log^{3/2}n}\right)\right) 
$$
for all $n\geq2$.
\end{theorem}

Since the second term in $h_{\mathcal C}(n!)$ turns out to be negative, it is not clear to us whether the Cram\'er model gives the correct prediction for this problem. It would be interesting to apply other models for the primes, such as the Cram\'er–Granville model (see, e.g., \cite[Chapter 17]{D}), to this problem, but this seems to be more difficult.

\smallskip

Other integers of interest are the values of the Euler totient function $\phi(n)$. It is well-known that a typical integer $n$ has about $\log\log n$ prime divisors, and, moreover, if one chooses a positive integer $n\leq x$ uniformly at random, then $\o(n)=\sum_{p|n}1$ behaves like a Poisson random variable with parameter $\log\log x$ (see Section \ref{sec4} and Theorem \ref{th4.2} for more precise statements). In particular, there is a version of the Central Limit Theorem for $\o$, which was first proved by Erd\H{o}s and Kac (\cite{EK} and \cite[Theorem 15.1]{D}): for every $u\in\R$,
$$
\lim_{x\to\infty}x^{-1}\#\left\{n\leq x: \o(n)-\log\log x\leq u(\log\log x)^{1/2}\right\} = N(u),
$$
where we write $N(u)=\frac{1}{\sqrt{2\pi}}\int_{-\infty}^ue^{-t^2/2}dt.$ 
In contrast, Erd\H{o}s and Pomerance \cite[Theorem 3.2]{EP85} showed that, for every $u\in\R$,
$$
\lim_{x\to\infty}x^{-1}\#\left\{n\leq x: \o(\phi(n))-\frac12(\log\log x)^2\leq \frac{u}{\sqrt3}(\log\log x)^{3/2}\right\} = N(u).
$$ 
Note that there is no chance to apply Erd\H{o}s-Kac-type results to $h(\phi(n))$, since $h$ is not additive. 

We establish the following result.

\begin{theorem}\label{th1.3} 
We have
$$ 
\sum_{n\leq x}h(\phi(n))\asymp  x\left(\frac{\log\log x}{\log\log\log x}\right)^{1/2}.
$$
\end{theorem}	

Our key ingredient in the proof of Theorem \ref{th1.3} is the Poisson random model for $\o$ (Theorem \ref{th4.2} below), introduced in a recent work of Ford \cite{Ford20}. We note that this model also played a crucial role in the recent works on sets whose differences avoid squares \cite{FG} and on primitive roots \cite{FGG}.

\medskip

\textbf{Notation.} Most of our notation is standard. By $\o(n), \tau(n), \s(n)$ we denote the number of prime divisors of $n$, the number of divisors of $n$, and the sum of divisors of $n$, respectively. The Euler totient function is denoted by $\phi(n)$. We use Vinogradov's notation $f\ll g$ and Landau's notation $f=O(g)$ in their usual sense. Letters $p$ and $q$, with or without subscripts, always denote prime numbers. The indicator function of an event $E$ is denoted by $\mathbf{1}_E$. We write $p^l\parallel n$ if $p^l| n$ but $p^{l+1}\nmid n$. We denote by $\Pois(\l)$ a Poisson random variable with parameter $\l$, and write $Z= \Pois(\l)$ for the statement that $Z$ is a Poisson random variable with parameter $\l$. We use the notation $\log_k x$ for the $k$-th iterate of the logarithm, that is, $\log_2x=\log\log x, \log_3x=\log\log\log x$, etc. By $c_1,c_2,\ldots$ we denote some positive absolute constants which appear in the text.

%In Sections 2,3,4...

\smallskip 

\textbf{Acknowledgments.} The study by Vitalii Iudelevich was implemented in the framework of the Basic Research Program at HSE University (HSE-BR-2025-84).

%We will also need the relations
%\begin{align}\label{1.10}
%	& \max_{n\leq x}h(n)  =(2+o(1))\left(\frac{\log x}{\log\log x}\right)^{1/2}, \\
%	& \max_{n\leq x}h\big(\phi(n)\big)  \asymp \left(\frac{\log x}{\log\log x}\right)^{1/2}. \label{1.11}
%\end{align} 
%and prove them in the Appendix.

\section{Proof of Theorem \ref{th1.1}} \label{sec2}

%Define 
%\begin{equation*}
%H(s)=\sum_{n=1}^{\infty}\frac{h(n)}{n^s}	
%\end{equation*}
%which is analytic in $\Re s>1$ by (\ref{1.3}). 

For an integer $\mu\geq0$, we define the functions
\begin{align*}
	f_{\mu}(n,s)&=\prod_{p|n}\frac{1-p^{-s}}{1-p^{-(\mu+1)s}}, \\
	H_{\mu}(s)&=\sum_{\substack{n=1\\ p|n\ro p^{\mu+1}|n}}^{\infty}\frac{\mu+h(n)}{n^s}f_{\mu}(n,s),\\
	F_{\mu}(s)&=\sum_{p|d\ro p^{\mu+1}|d}d^{-s}f_{\mu-1}(d,s).
\end{align*}
We will use the method of complex integration and, to extract the first terms in the asymptotic of $\sum_{n\leq x}h(n)$, will need to find the singularities of $H_0(s)=\sum_{n}h(n)n^{-s}$. In order to do so, one basically should write $n=abcd$, when $a$ is square-free, $b$ is the square of a square-free number, $c$ is the cube of a square-free number, and $p|d$ implies $p^4|d$, and exploit the identity 
$$
h(n)=h(d)+3-\1(a=1)-\1(b=1)-\1(c=1)
$$
when working with the quadruple sum over $a,b,c,d$. To keep it succinct, we instead perform one such step and then iterate the procedure given in the following lemma. 

\begin{lem}\label{lem2.1}
For $\mu\geq1$, 
\begin{equation*}
H_{\mu-1}(s)=\frac{\zeta(\mu s)}{\zeta((\mu+1)s)}H_{\mu}(s)-F_{\mu}(s).	
\end{equation*}	
\end{lem}

\begin{proof}
For a number $n$ from the summation in $H_{\mu-1}(s)$, we write $n=cd$, where $(c,d)=1$ and, for any prime $p$, $p|c$ implies $p^{\mu}||c$ and $p|d$ implies $p^{\mu+1}|d$. By the definition of $h$,
	$$
	h(n)=h(d)+1-\1(c=1)
	$$
	and	thus
	\begin{align*}
		H_{\mu-1}(s)&=\sum_{\substack{c,d:\, (c,d)=1,\\ p|c\ro p^{\mu}||c,\\p|d\ro p^{\mu+1}|d}}\frac{\mu+h(d)-\1(c=1)}{(cd)^s}f_{\mu-1}(cd,s)\\
		&=\sum_{p|d\ro p^{\mu+1}|d}\frac{\mu+h(d)}{d^s}f_{\mu-1}(d,s)\sum_{\substack{(c,d)=1:\\ p^{\mu}||c}}c^{-s}f_{\mu-1}(c,s)-F_{\mu}(s)\\
		&=\sum_{p|d\ro p^{\mu+1}|d}\frac{\mu+h(d)}{d^s}f_{\mu-1}(d,s)\prod_{p\notdivides d}\left(1+p^{-\mu s}\frac{1-p^{-s}}{1-p^{-\mu s}}\right)-F_{\mu}(s).
	\end{align*}
	Direct calculations show that 
	$$
	f_{\mu-1}(d,s)\prod_{p\notdivides d}\left(1+p^{-\mu s}\frac{1-p^{-s}}{1-p^{-\mu s}}\right)=\frac{\zeta(\mu s)}{\zeta((\mu+1)s)}f_{\mu}(d,s),
	$$
	and the claim follows. 
\end{proof}	

Now we prepare to complex integration and compute the residues of the function
\begin{equation}\label{2.1}
G(s)=H_0(s)\frac{x^s}{s}.	
\end{equation}
We recall the definition (\ref{1.3}) of the constant $c_{\h}$.

\begin{lem}\label{lem2.2}
Let $x\geq1$. Then
\begin{itemize}
\item[(i)] $\res_{s=1}G(s)=c_{\h}x$;

\item[(ii)]	$\res_{s=1/2}G(s)=-\frac{\zeta(3/2)}{\zeta(3)}x^{1/2}$;
		
\item[(iii)] $\res_{s=1/3}G(s)=-\left(\frac{\zeta(2/3)}{\zeta(2)}+\frac{\zeta(1/3)}{\zeta(2/3)}\prod_{p}\left(1+\frac{1}{p^{5/3}+p^{4/3}}\right)\right)x^{1/3}$.	 
	\end{itemize}
\end{lem}

\begin{proof} The previous lemma gives
\begin{equation}\label{2.2}
H_0(s)=\frac{\zeta(s)}{\zeta(2s)}H_1(s)-F_1(s),
\end{equation}
and since $\zeta(s)=(1+o(1))(s-1)^{-1}$ as $s\to1$ and both $H_1(s)$ and $F_1(s)$ are analytic in the region $\Re s>1/2$, we have 	
	$$
	\res_{s=1}G(s)=\frac{1}{\zeta(2)}H_1(1)x=c_{\h}x.
	$$	
	This proves (i). Next, applying Lemma \ref{lem2.1} one more time, from (\ref{2.1}) we get
	\begin{equation}\label{2.3}
		H_0(s)=\frac{\zeta(s)}{\zeta(3s)}H_2(s)-\frac{\zeta(s)}{\zeta(2s)}F_2(s)-F_1(s).
	\end{equation}
	The first two terms here are analytic in, say, a small neighborhood	of the point $s=1/2$. Also, 
	\begin{equation}\label{2.4}
		F_1(s)=\frac{\zeta(2s)\zeta(3s)}{\zeta(6s)}.	
	\end{equation}
	Therefore
$$
\res_{s=1/2}G(s)=-2x^{1/2}\res_{s=1}F_1(s)=-\frac{\zeta(3/2)}{\zeta(3)}x^{1/2},
$$
and (ii) follows. Finally, by another application of Lemma \ref{lem2.1} we deduce from (\ref{2.3}) that
$$
H_0(s)=\frac{\zeta(s)}{\zeta(4s)}H_3(s)-\frac{\zeta(s)}{\zeta(3s)}F_3(s)-\frac{\zeta(s)}{\zeta(2s)}F_2(s)-F_1(s).
$$
The first two terms are analytic in a neighborhood of the point $s=1/3$, so that
$$
\res_{s=1/3}G(s)=-3x^{1/3}\left(\frac{\zeta(1/3)}{\zeta(2/3)}\res_{s=1/3}F_2(s)-\res_{s=1/3}F_1(s)\right).
$$
Now (iii) follows from
\begin{equation}\label{2.5}
F_2(s)=\prod_{p}\left(1+\frac{1}{p^{3s}-p^s}\right)=\zeta(3s)\prod_{p}\left(1+\frac{1}{p^{4s}(p^s+1)}\right) 
\end{equation}
and (\ref{2.4}). This completes the proof of the lemma.
\end{proof}

We are now ready to prove Theorem \ref{th1.1}. Let $x$ and $T\leq x$ be large enough. We will need the Vinogradov-Korobov zero-free region for $\zeta(s)$ (see, for instance, \cite[Corollary 8.28, Theorem 8.29]{IK}): there is an absolute constant $c_0>0$ such that $\zeta(s)\neq0$ for $s=\s+it$ with
\begin{equation}\label{2.6} 
	\s\geq 1-2\d, \quad \d=\frac{c_0}{(\log T)^{2/3}(\log\log T)^{1/3}};
\end{equation}
also, in this region one has
\begin{equation}\label{2.7} 
	\zeta(s)\ll (\log T)^{2/3}, \quad \frac{1}{\zeta(s)}\ll (\log T)^{2/3}(\log_2 T)^{1/3}.	
\end{equation}
It is easy to deduce that
\begin{equation}\label{2.75}
\max_{n\leq x}h(n) \asymp \left(\frac{\log x}{\log\log x}\right)^{1/2}.
\end{equation}
By this estimate and Perron's formula (see \cite[Theorem 7.2]{D}), applied with $\theta=0, A=1/2$, $\a=1+1/\log x$, we deduce (recall the definition (\ref{2.1}) of $G$)
\begin{equation*}	
\sum_{n\leq x}h(n)=\frac{1}{2\pi i}\int_{\a-iT}^{\a+iT}G(s)ds+O\left(\frac{x(\log x)^{3/2}}{T}\right).	
\end{equation*} 
Let $\beta=1/2-\d$ and consider the contour $L$ with the vertices $\a\pm iT, \b\pm iT$. By (\ref{2.2}) and (\ref{2.4}) we see that $G(s)$ have only two poles inside $L$ at the points $s=1$ and $s=1/2$. Cauchy's residue theorem and Lemma \ref{lem2.2} give 
$$
\frac{1}{2\pi i}\int_{L}G(s)ds=c_{\h}x-\frac{\zeta(3/2)}{\zeta(3)}x^{1/2}.
$$ 
Putting everything together we find 
\begin{equation}\label{2.8}
	\sum_{n\leq x}h(n)=c_{\h}x-\frac{\zeta(3/2)}{\zeta(3)}x^{1/2}+R,
\end{equation}
where
\begin{equation}\label{2.9}
	R=I_1+I_2-I_3+O\left(\frac{x(\log x)^{3/2}}{T}\right)	
\end{equation} 
and 
\begin{align*}
I_1=\int_{\b+iT}^{\a+iT}G(s)ds, \quad I_2=\int_{\b-iT}^{\b+iT}G(s)ds, 
\quad I_3=\int_{\b-iT}^{\a-iT}G(s)ds.	
\end{align*}
Thus, it suffices to estimate the integrals $I_i$, $i=1,2,3$. From (\ref{2.3})-(\ref{2.5}) we have 
$$
\max\{|I_1|,|I_3|\} \ll \int_{\b}^{\a}\left(|\zeta(\s+iT)|+\frac{|\zeta(\s+iT)|}{|\zeta(2\s+2iT)|}+|\zeta(2\s+2iT)|\right)\frac{x^{\s}}{T}d\s. 
$$
It is well-known (see \cite[Theorem 1.9]{Ivic}) that  
$$ 
|\zeta(\s+iT)|\ll \left(T^{\frac{1-\s}{2}}+1\right)\log T 
$$
for $0\leq \s\leq 2$. From here and (\ref{2.7}) we get
\begin{equation}\label{2.10}
	\max\{|I_1|,|I_3|\} \ll	\frac{(\log T)^2}{T}\int_{\b}^{\a}\left(T^{\frac{1-\s}{2}}+1\right)x^{\s}d\s \ll \frac{x\log x}{T}.
\end{equation}
Further,
\begin{align*}
	I_2 &\ll \int_0^T\left(|\zeta(\b+it)|+\frac{|\zeta(\b+it)|}{|\zeta(2\b+2it)|}+|\zeta(2\b+2it)|\right)\frac{x^{\beta}}{1+t}dt \\
	& \ll x^{1/2-\d}(\log x)\left(1+\int_1^T\frac{|\zeta(1/2-\d+it)|}{t}dt\right). 	
\end{align*}	 
From \cite[equation (6.8)]{D}, we can write
$$ 
|\zeta(1/2-\d+it)| \asymp |t|^{\d}|\zeta(1/2+\d+it)|, \quad |t|\geq1. 
$$
Thus,
\begin{equation}\label{2.11}
	I_2 \ll x^{1/2}(T/x)^{\d}(\log x)\int_1^T\frac{|\zeta(1/2+\d+it)|}{t}dt + x^{1/2-\d}\log x.	
\end{equation} 
Now one has (see \cite[Theorem 7.2]{TH}) 
$$
\int_1^U|\zeta(\s+it)|^2dt \ll U\min\Big\{\log U, \frac{1}{\s-1/2}\Big\}
$$ 
uniformly for $U\geq1$, $1/2\leq \s\leq 2$, and hence, for $1\leq U\leq T$, 
$$ 
\int_U^{2U}\frac{|\zeta(1/2+\d+it)|}{t}dt \ll (\log T)^{1/2}
$$ 
by the Cauchy-Schwarz inequality. From (\ref{2.11}) it follows that
$$
I_2\ll x^{1/2}(T/x)^{\d}(\log x)^{7/2}.
$$
From here, (\ref{2.9}) and (\ref{2.10}) we have
$$
R \ll \frac{x(\log x)^{3/2}}{T}+x^{1/2}(T/x)^{\d}(\log x)^{7/2}.
$$
By taking, say, $T=x^{3/4}$, we get
$$
R \ll x^{1/2}\exp\left(-c_1(\log x/\log_2x)^{1/3}\right)
$$ 
for some absolute $c_1>0$, and Theorem \ref{th1.1} follows from (\ref{2.8}).

\section{Estimating $h(n!)$} \label{sec3}

\subsection{Proof of the relation (\ref{1.4})}

In this subsection, for completeness, we provide a short proof of (\ref{1.4}), since later we anyway need Lemma \ref{lem3.1} below and Erd\H{o}s' paper \cite{Erd82} is not very detailed.

%We start with the proof of (\ref{1.6}). 

We may assume that $n$ is large enough. Let $\{a_i\}_{i=1}^m$ and $\{b_i\}_{i=1}^m$ be two non-increasing sequences. Note that $a_{i_j}>a_{i_{j+1}}$ implies $a_{i_j}+b_{i_j}>a_{i_{j+1}}+b_{i_{j+1}}$, and $a_{i_j}+b_{i_j}>a_{i_{j+1}}+b_{i_{j+1}}$ implies either $a_{i_j}>a_{i_{j+1}}$ or $b_{i_j}>b_{i_{j+1}}$. This gives 
\begin{equation}\label{3.-1}
\#\{a_i\} \leq \#\{a_i+b_i\} \leq \#\{a_i\}+\#\{b_i\}.	
\end{equation}
Now if denote $a_p=\lfloor \frac{n}{p}\rfloor$ and $b_p=\sum_{j\geq2}\lfloor\frac{n}{p^j}\rfloor$, we have 
$$
\#\{b_p: p\leq n\} \leq \pi(n^{1/3})+\#\{b_p: n^{1/3}<p\leq n\} \ll n^{1/3}.
$$
Since $n!=\prod_{p\leq n}p^{a_p+b_p}$, from (\ref{3.0}) we deduce
\begin{equation}\label{3.0}
h(n!)=\#\left\{\left\lfloor\frac{n}{p}\right\rfloor: p\leq n\right\}+O(n^{1/3}).		
\end{equation} 
% Write $n!=N_1N_2$, where $P^+(N_1)\leq (n\log n)^{1/2}$ and $P^-(N_2)>(n\log n)^{1/2}$. Then 
From here we trivially have
$$
h(n!)\leq \pi\big((n\log n)^{1/2}\big)+\#\left\{\left\lfloor\frac{n}{p}\right\rfloor: p>(n\log n)^{1/2}\right\}+O(n^{1/3}) \ll (n/\log n)^{1/2},
$$
%\begin{align*} 
%h(N_1) &\leq \o(N_1) \ll (n/\log n)^{1/2}, \\
%h(N_2) &\ll \max_{p>(n\log n)^{1/2}}\frac{n}{p} \leq (n/\log n)^{1/2},
%\end{align*}
and the upper bound in (\ref{1.4}) follows.

% since $h(ab)\leq h(a)+h(b)$ for coprime $a$ and $b$. 

Now we turn to the lower bound. It is a consequence of the following standard lemma.  

\begin{lem}\label{lem3.1}
Let $x$ be large and $2=p_1<p_2<\ldots$ be the consecutive primes. Then
$$
\#\{p_i\leq x: p_i-p_{i-1}>0.01\log x\} \geq 0.5\pi(x).
$$	
\end{lem}

\begin{proof} It is well-known that the upper sieve gives the estimate
$$
\#\{p\leq x: p+a \mbox{ prime} \} \leq \frac{10a}{\varphi(a)}\frac{x}{\log^2x}
$$ 
uniformly in $a$ (see, for instance, \cite[Corollary 21.3]{D}). Also, $a/\varphi(a)\leq (\pi^2/6)\sum_{d|a}d^{-1}$ and
$$
\sum_{a\leq x}\frac{a}{\varphi(a)} \leq \frac{\pi^2}{6}\sum_{d\leq x}\frac{x}{d^2} \leq 3x. 
$$
Hence,
\begin{align*} 
\#\{p\leq x: \exists a\leq 0.01\log x \mbox{ with } p+a \mbox{ prime}\} \leq 
\sum_{a\leq 0.01\log x}\#\{p\leq x: p+a \mbox{ prime} \} \leq 0.5\pi(x),
\end{align*}
and the claim follows.
\end{proof}

Let $x=0.01(n\log n)^{1/2}$ and $q_1<\ldots<q_s$ be all the primes $p_i$ up to $x$ which satisfy the property from Lemma \ref{lem3.1}. Then 
$$
\left\lfloor\frac{n}{q_i}\right\rfloor-\left\lfloor\frac{n}{q_{i+1}}\right\rfloor\geq \frac{n(q_{i+1}-q_i)}{q_iq_{i+1}}-1 \geq \frac{0.01n\log x}{x^2}-1\geq 48,
$$
provided that $x$ is large. This means that the numbers $\lfloor n/q_1\rfloor,\ldots,\lfloor n/q_s\rfloor$ are distinct. Since $s\gg \pi(x) \gg (n/\log n)^{1/2}$, the lower bound in (\ref{1.4}) follows from (\ref{3.0}).

\subsection{Proof of Theorem \ref{th1.2}} Recall the relation (\ref{1.9}) and the indicator functions $A_k$. We set
\begin{equation}\label{3.1}
L_1=\lfloor n^{1/2}(\log n)^{-10}\rfloor, \quad L_2=n^{1/2}, \quad L_3=\lfloor n^{1/2}(\log n)^{10}\rfloor. 	
\end{equation}
First, we always have
$$
\sum_{k>L_3}A_k \leq \#\{l\in \C: l\leq n/L_3 \},
$$
which gives
\begin{equation*}
\sum_{k>L_3}\E A_k \ll \frac{n}{L_3\log n}. 
\end{equation*}
Also, trivially
\begin{equation*}
\sum_{k\leq L_1}\E A_k \leq L_1.
\end{equation*}
Thus,
\begin{equation}\label{3.2}
\E h_{\C}(n!)=\sum_{L_1<k\leq L_3}\E A_k +O(n^{1/2}(\log n)^{-10}).	
\end{equation}
Now let $L_1 \leq k \leq L_3$ and $l\in I_k$. Then $l=n/k+O(n/k^2)$, $\log l=\log(n/k)+O(1/k)$ and 
$$
\frac{1}{\log l}=\frac{1}{\log(n/k)}+O\left(\frac{1}{k\log^2n}\right),
$$
which gives
\begin{equation}\label{3.3}
1-\frac{1}{\log l}=\left(1-\frac{1}{\log(n/k) }\right)\left(1+O\left(\frac{1}{k\log^2n}\right)\right).
\end{equation}
Also, we define
$$
a_k=|I_k\cap \Z|
$$
and note that 
\begin{equation}\label{3.4}
\sum_{A<k\leq B}a_k=\frac{n}{A+1}-\frac{n}{B+1}+O(1)
\end{equation}
for $n^{1/2}\leq A<B\leq n$, and that
\begin{equation}\label{3.5}
a_k=nk^{-2}+b_k,
\end{equation}
where $b_k=-\frac{n}{k^2(k+1)}-\left\{\frac{n}{k}\right\}+\left\{\frac{n}{k+1}\right\}=O(1)$ obey
\begin{equation}\label{3.6}
\sum_{L_1<k\leq t}b_k \ll nL_1^{-2}=(\log n)^{20}	
\end{equation}
uniformly for $t>L_1$. Thus, (\ref{1.8}) together with (\ref{3.3}) gives
\begin{align*}
\P\left(I_k \cap \C=\varnothing \right)	= \prod_{l\in I_k}\left(1-\frac{1}{\log l}\right)=\left(1-\frac{1}{\log(n/k) }\right)^{a_k}+O\left(\frac{a_k}{k\log^2n}\right),
\end{align*}
so that
\begin{equation*}
\E A_k=1-\left(1-\frac{1}{\log(n/k) }\right)^{a_k}+O\left(\frac{a_k}{k\log^2n}\right)
\end{equation*}	
for $L_1\leq k\leq L_3$. For $k\geq L_2=n^{1/2}$, we have $a_k\ll 1$ and $a_k^2\ll a_k$ by (\ref{3.5}). Therefore,
$$ 
\E A_k = \frac{a_k}{\log(n/k)}+O\left(\frac{a_k}{\log^2n}\right)
$$
and
$$
\sum_{L_2<k\leq L_3}\E A_k=\sum_{L_2\leq k\leq L_3}\frac{a_k}{\log(n/k)}+O\left(\frac{n}{L_2\log^2n}\right).
$$
Partial summation with 
\begin{equation}\label{3.7}
f(t)=\frac{1}{\log(n/t)}=\frac{1}{\log n}+\frac{\log t}{\log^2n}+O\left(\frac{\log^2t}{\log^3n}\right)	
\end{equation}
and (\ref{3.4}) imply (recall the definition (\ref{3.1}) of the numbers $L_i$)
\begin{align*}
\sum_{L_2<k\leq L_3}\frac{a_k}{\log(n/k)}&=\frac{nf(L_3)}{L_2+1} + n\int_{L_2}^{L_3}\frac{dt}{t(t+1)\log^2(n/t)}-\frac{n}{L_2+1}(f(L_3)-f(L_2))+O\left(\frac{nf(L_3)}{L_3}\right)\\
&=\frac{2n^{1/2}}{\log n}+O\left(\frac{n^{1/2}}{\log^2n}\right),
\end{align*}
since $f'(t)>0$ for $t>0$. Therefore, 
\begin{align}\label{3.8}
\sum_{L_2<k\leq L_3}\E A_k=\frac{2n^{1/2}}{\log n} +O\left(\frac{n^{1/2}}{\log^2 n}\right). 
\end{align}
For $L_1<k\leq L_2$, we write 
$$
l_k=\log\left(1-\frac{1}{\log(n/k)}\right) 
$$
for short. Note that $l_k<0$ and $|l_k| \asymp \frac{1}{\log n}$. By (\ref{3.5}),
$$
\left(1-\frac{1}{\log(n/k)}\right)^{a_k}=e^{nl_k/k^2+b_kl_k}=e^{nl_k/k^2}\left(1+b_kl_k+O\left(\frac{1}{\log^2n}\right)\right)
$$
and hence
\begin{align}\label{3.9}
\sum_{L_1<k\leq L_2}\E A_k=\sum_{L_1<k\leq L_2}\left(1-e^{nl_k/k^2}\right)-\sum_{L_1< k\leq L_2}b_kl_ke^{nl_k/k^2}+O\left(\frac{L_2}{\log^2n}\right).
\end{align}
It is easy to check that the smooth function $l_ke^{nl_k/k^2}$ has negative derivative with respect to the variable $k$ in the range $L_1\leq k\leq L_2$, and therefore partial summation together with (\ref{3.6}) implies 
$$
\sum_{L_1<k\leq L_2}b_kl_ke^{nl_k/k^2} \ll (\log n)^{20}.
$$
Further, if we let 
$$
g(t)=1-e^{nl_t/t^2},$$
the Euler-Maclaurin formula gives 
$$
\sum_{L_1<k\leq L_2}\left(1-e^{nl_k/k^2}\right)=\int_{L_1}^{L_2}g(t)\,dt-\int_{L_1}^{L_2}g'(t)\big(1/2-\{t\}\big)dt+O(1).
$$
The second integral here is $O(1)$, because $g'(t)<0$ for $L_1\leq t\leq L_2$. Putting everything together, we see from (\ref{3.9}) that
\begin{align*}
\sum_{L_1<k\leq L_2}\E A_k &=\int_{L_1}^{L_2}g(t)\,dt+O\left(\frac{L_2}{\log^2 n}\right) \\
&=n\int_{n/L_2}^{n/L_1}\left(1-\left(1-\frac{1}{\log u}\right)^{u^2/n}\right)u^{-2}du+O\left(\frac{L_2}{\log^2 n}\right). 
\end{align*}
Making the change of variables $u=v^{1/2}n^{1/2}$ and estimating the contribution from $u\geq (\log n)^{3/2}$ to the resulting integral trivially, we find
\begin{equation}\label{3.10}
\sum_{L_1<k\leq L_2}\E A_k=n^{1/2}-0.5n^{1/2}I+O\left(\frac{n^{1/2}}{\log^2 n}\right),	
\end{equation}
where
\begin{align*}
I=\int_{1}^{(\log n)^{3/2}}\left(1-\frac{2}{\log n+\log v}\right)^{v}v^{-3/2}dv.
\end{align*}
For $1\leq v\leq (\log n)^{3/2}$, we have
$$
\frac{1}{\log n+\log v}=\frac{1}{\log n}-\frac{\log v}{\log^2n}+O\left(\frac{\log^2v}{\log^3n}\right)
$$
and thus, using $\log(1-x)=-x-x^2/2+O(x^3)$ for $0<x\leq1$, 
$$
\log\left(1-\frac{2}{\log n+\log v}\right)-\log\left(1-\frac{2}{\log n}\right)=\frac{2\log v}{\log^2n}+O\left(\frac{\log^2v+1}{\log^3n}\right). 
$$
This gives
\begin{align}\label{3.11}
I&=\int_1^{(\log n)^{3/2}}v^{-3/2}e^{-av}\left(1+\frac{2v\log v}{\log^2n}+O\left(\frac{v\log^2(2v)}{\log^3n}\right)\right)dv\\
&=I_1+\frac{2I_2}{\log^2n}+O\left(\frac{I_3}{\log^3n}+e^{-\sqrt{\log n}}\right), \nonumber
\end{align}
where we denote 
$$
a=-\log\left(1-\frac{2}{\log n}\right)=\frac{2}{\log n}+\frac{2}{\log^2n}+O\left(\frac{1}{\log^3n}\right)
$$
and
\begin{align*}
I_1=\int_1^{\infty}v^{-3/2}e^{-av}dv, \quad 
I_2=\int_1^{\infty}v^{-1/2}(\log v)e^{-av}dv, \quad
I_3=\int_1^{\infty}v^{-1/2}(\log^2 (2v))e^{-av}dv.	
\end{align*}
It is easy to see that 
\begin{equation}\label{3.12}
I_3\ll (\log n)^{1/2}(\log\log n)^2.	
\end{equation} 
Now, for $x>0$ and a real non-integer $z$, let
$$
\G(z,x)=\int_x^{\infty}u^{z-1}e^{-u}du 
$$
be the upper incomplete gamma function. Integrating by parts, we see that
$$
\G(z+1,x)=z\G(z,x)+x^{z}e^{-x}.
$$
Then
\begin{equation*}
I_1=a^{1/2}\G(-1/2,a)=2e^{-a}-2a^{1/2}\G(1/2,a),
\end{equation*} 
and
\begin{equation}\label{3.13}
\G(1/2,a)=\G(1/2)-\int_0^au^{-1/2}e^{-u}du=\sqrt{\pi}-2a^{1/2}+O(a^{3/2}).
\end{equation}
Hence,
\begin{align}\label{3.14}
I_1&=2-\frac{4}{\log n}-2\sqrt{\pi}a^{1/2}+4a+O(a^2)\\
&=2-\frac{2\sqrt{2\pi}}{(\log n)^{1/2}}+\frac{4}{\log n}-\frac{\sqrt{2\pi}}{(\log n)^{3/2}}+O((\log n)^{-2}). \nonumber
\end{align}
Further,
\begin{align*}
I_2&=a^{-1/2}\int_a^{\infty}u^{-1/2}(\log u-\log a)e^{-u}du\\
&=a^{-1/2}\log(1/a)\G(1/2,a)+a^{-1/2}\G'(1/2)-a^{-1/2}\int_0^au^{-1/2}e^{-u}\log u\, du
\end{align*}
The last integral can be seen to be $\ll a^{1/2}\log\log n$. Since $\log(1/a)=\log\log n-\log2+O(1/\log n)$, from (\ref{3.13})  we get
\begin{align}\label{3.15}
I_2&=\sqrt{\pi}a^{-1/2}\log_2n+a^{-1/2}(\G'(1/2)-\sqrt{\pi}\log2)+O(\log_2n)\\
&=\sqrt{\frac{\pi}{2}}(\log n)^{1/2}\log_2n+c_2(\log n)^{1/2}+O(\log_2n),	\nonumber
\end{align}
where 
$$
c_2=\frac{\G'(1/2)-\sqrt{\pi}\log2}{\sqrt2}=-\sqrt{\frac{\pi}{2}}(\g+\log8)
$$
and $\g=0.577\ldots$ is the Euler constant. From (\ref{3.11}), (\ref{3.12}), (\ref{3.14}), (\ref{3.15}), we deduce
$$
I=2-\frac{2\sqrt{2\pi}}{(\log n)^{1/2}}+\frac{4}{\log n}+2\sqrt{\frac{\pi}{2}}\frac{\log_2n}{(\log n)^{3/2}}+\frac{2c_2-\sqrt{2\pi}}{(\log n)^{3/2}}+O\left(\frac{\log_2n}{\log^2n}\right).
$$
From here, combining (\ref{3.2}), (\ref{3.8}), (\ref{3.10}), we finally get
$$
\E h_{\C}(n!)=\frac{\sqrt{2\pi}n^{1/2}}{(\log n)^{1/2}}\left(1-\frac{\log_2n}{2\log n}+\frac{1+\g+\log8}{2\log n}+O\left(\frac{\log_2n}{\log^{3/2}n}\right)\right). 
$$
Now 
$$
\Var h_{\C}(n!)=\sum_{k\geq1}\left(\P(A_k)-\P(A_k)^2\right) \leq \E\, h_{\C}(n!) \ll \left(\frac{n}{\log n}\right)^{1/2}.
$$
A variant of Chernoff's inequality (e.g., \cite[Theorem 1.8]{TV})  then gives
$$
\P\left(|h_{\C}(n!)-\E\,h_{\C}(n!)|\geq n^{1/3}\right) \ll \exp(-0.25n^{1/6}),
$$
and an application of the Borel-Cantelli lemma completes the proof.

\section{Average value of $h(\phi(n))$} \label{sec4}

In this section we prove Theorem \ref{th1.3}. We start with some heuristic argument. For a subset $T$ of primes in $[2,x]$, set
$$
\o(n,T)=\#\{p|n: p\in T\}.
$$
Since most of $n\leq x$ have small square-full part, $h(\phi(n))$ is usually close to $h(l(n))$, where $l(n)=\prod_{p|n}(p-1)$. Let us denote by $v_q(n)$ the exponent in which a prime $q$ appears in $l(n)$. We have
\begin{equation}\label{4.1}
v_q(n)=\sum_{j\geq1}\o(n,T_{q^j}),	
\end{equation}
where $T_u=\{p\leq x: p\equiv 1\pmod{u}\}$. Let us discard the terms with $j\geq2$ for now, since they would not contribute much on average over $n\leq x$; thus, a natural approximation is 
$$
v_q(n)\approx \o(n,T_q).
$$ 
It is known that, when $n\leq x$ is drawn uniformly at random, one can treat $\o(n,T)=\#\{p|n: p\in T\}$ as the Poisson random variable with the parameter 
\begin{equation}\label{4.2}
H(T)=\sum_{p\in T}\frac1p \, .
\end{equation} 
Moreover, there is a powerful generalization of this due to Ford \cite{Ford20} (see also \cite{Ten}), which states that, provided that $T_1,\ldots, T_r$ are disjoint subsets of primes, the vector $(\o(n,T_1),\ldots, \o(n,T_r))$ has distribution close to that of the vector $(Z_1,\ldots,Z_r)$, where $Z_i$ are independent Poisson random variables with parameters $H(T_i)$ (see Theorem \ref{th4.2} below). Using the Mertens theorem for arithmetic progressions, we then expect 
$$
v_q(n) \approx \Pois\left(\frac{\log_2x}{\phi(q)}\right). 
$$ 
%with $H(T_q)=\frac{\log_2x}{\phi(q)}$. 
Hence, most of the time, large primes $q$ will give small value of $v_q(n)$, and trivially, there are just not too many small primes $q$. Thus, the main task will be to handle $\o(n,T_q)$ for medium-size primes $q$; as we will see, the critical range here is $q \asymp M:=(\log_2x\log_3x)^{1/2}$. 

To establish the upper bound in Theorem \ref{th1.3}, one needs to show that all (or at least most) values of $v_q(n)$ with $q>M$ are usually $\ll \pi(M)$. To do so, it is enough to use estimates for the tails of Poisson distribution for individual $q$ and then simply exploit the union bound. In fact, the crude elementary estimate
\begin{equation}\label{4.3}
\#\{n\leq x: \o(n,T)=k\}\leq \frac{1}{k!}\sum_{p_1,\ldots,p_k\in T}\frac{x}{p_1\ldots p_k} \leq \frac{xH(T)^k}{k!}.
\end{equation}
will suffice for this purpose. For the lower bound, we will show that, for most $n\leq x$, a positive proportion of $v_q(n)$ with $q\asymp cM$ are distinct, provided that $c>0$ is small enough. Despite the corresponding sets $T_q$ are not disjoint, the random variables $\o(n,T_q)$ are still close to be independent because of the relatively small correlations between them (this corresponds to the fact that pairwise intersections of the sets $T_q$ are much smaller than these sets themselves). Theorem \ref{th4.2} below then allows us to reduce the initial problem to a purely probabilistic one, and we are able to conclude that the event that there are many of collisions among the values of $\Pois(H(T_q))$ is rare via some combinatorial argument.

Now we turn to the details. For $n=p_1^{\a_1}\ldots p_s^{\a_s}$, we write
$$
\phi(n)=l(n)l'(n),
$$
where 
\begin{align*}
l(n)=(p_1-1)\ldots (p_s-1), \quad 
l'(n)=p_1^{\a_1-1}\ldots p_s^{\a_s-1}.	
\end{align*}
The simple inequality
$$
\left|h(ab)-h(a)\right|\leq \o(b)
$$
implies
\begin{equation}\label{4.4}
\sum_{n\leq x}h(\phi(n))=\sum_{n\leq x}h(l(n))+O\left(\sum_{n\leq x}\o(l'(n))\right)=\sum_{n\leq x}h(l(n))+O(x\log_3x); 	
\end{equation}
here in the second step we used that $P^+(l'(n))\leq (\log x)^{10}$ for all but $O(x(\log x)^{-10})$ numbers $n\leq x$ (together with $\o(n)\ll \log n$) and $\sum_{n\leq x}\#\{p|n: p\leq t\}\ll x\log_2t$ for any $t\geq 3$. 

 The following lemma states that the main contribution to the sum in the right side of (\ref{4.4}) comes from the numbers $n\leq x$ with $h(l(n))\ll (\log_2x)^2$.

\begin{lem}\label{lem4.1}
We have
$$
\#\{n\leq x: \o(l(n))>100(\log_2x)^2\} \ll x(\log x)^{-10}. 
$$	
\end{lem}

\begin{proof} First, an application of (\ref{4.3}) with $T$ equals to the set of all primes not exceeding $x$ gives 
\begin{equation*}
\#\{n\leq x: \o(n)=k\}\leq \frac{x(\log_2x+O(1))^k}{k!}.
\end{equation*}
This implies
$$	
\#\{n\leq x: \o(n)>10\log_2x\} \ll x(\log x)^{-13} 	
$$
and similarly, for $t\leq x$,
$$
\#\{p\leq t: \o(p-1)>10\log_2x\} \leq \#\{n\leq t: \o(n)>10\log_2x\} \ll t(\log x)^{-13}.
$$
Then
\begin{align*}
& \#\{n\leq x: \exists p|n \mbox{ with } \o(p-1)>10\log_2x\} \\
& \leq x\sum_{\substack{p\leq x: \\ \o(p-1)>10\log_2x}}\frac1p	
\ll x\int_{2}^x\frac{dt}{t(\log x)^{13}} \ll x(\log x)^{-10},
\end{align*}
and the claim now follows from $\o(l(n)) \leq \sum_{p|n}\o(p-1)$.
\end{proof}

\subsection{Upper bound} Let $x$ be large. By (\ref{2.75}), (\ref{4.4}) and Lemma \ref{lem4.1}, we see that to prove the upper bound in Theorem \ref{th1.3}, it suffices to show that
\begin{equation}\label{4.5}
\sum_{\substack{n\leq x \\ h(l(n)) \leq 100\log_2^2x}} h(l(n)) \ll x(\log_2x/\log_3x)^{1/2}.	
\end{equation}  
Define
$$
M=(\log_2x\log_3x)^{1/2},
$$
and, for $n\leq x$, let 
$$
l(n)=l_1(n)l_2(n),
$$ 
where 
$$
P^+(l_1(n))\leq M<P^-(l_2(n)).
$$ 
Then $h(l(n))\leq h(l_1(n))+h(l_2(n))$ and trivially
$$
h(l_1(n)) \leq \pi(M) \ll (\log_2x/\log_3x)^{1/2}.
$$
Thus, (\ref{4.5}) is reduced to
\begin{equation}\label{4.6}
\sum_{\substack{n\leq x \\ h(l(n)) \leq 100\log_2^2x}}h(l_2(n)) \ll x(\log_2x/\log_3x)^{1/2}.	
\end{equation}
Write
$$
l_2(n)=\prod_{q>M}q^{v_q(n)}.
$$
We claim that, for most $n\leq x$, each $v_q(n)$ here obeys 
\begin{equation}\label{4.7}
v_q(n) \ll U:=(\log_2x/\log_3x)^{1/2}.
\end{equation}
Fix a prime $q>M$. The inequality (\ref{4.3}) gives
\begin{equation}\label{4.8}
\#\{n\leq x: \o(n,T)= k\} \ll x(eH(T)/k)^k.
\end{equation} 
Also, trivially
\begin{equation}\label{4.9}
H(T_{q^j})=\int_{M}^x\frac{\pi(t,1,q^j)\,dt}{t^2}+O(q^{-j}) \ll \int_{M}^x\frac{dt}{tq^j}\leq \frac{\log x}{q^j}
\end{equation} 
for any modulus  $q^j$, and, by the Brun-Titchmarsh inequality,
\begin{equation}\label{4.10}
H(T_{q^j})\leq \int_{M}^{(\log x)^{11}}\frac{dt}{tq^j}+\int_{(\log x)^{11}}^x\frac{\pi(t,1,q^j)\,dt}{t^2}+O(q^{-j}) \ll \frac{\log_2x}{q^j}
\end{equation} 
if $q^j\leq (\log x)^{10}$. 

Let $q>M$. For $j\geq2$ with $q^j>(\log x)^{10}$, we use (\ref{4.8}) with $k\geq1$ and (\ref{4.9}) to obtain 
$$
\#\{n\leq x: \o(n,T_{q^j})>0\} \ll \frac{x\log x}{q^{j}}.
$$
For $q>(\log x)^{10}$, (\ref{4.8}) with $k\geq2$ and (\ref{4.9}) give
$$
\#\{n\leq x: \o(n,T_q)>1\} \ll \frac{x(\log x)^2}{q^2}.
$$
For $j\geq10$ with $q^j\leq (\log x)^{10}$, (\ref{4.8}) and (\ref{4.10}) imply
$$
\#\{n\leq x: \o(n,T_{q^j})>0\} \ll \frac{x\log_2 x}{q^j}.
$$
Finally, for $j\leq10$ with $q^j\leq (\log x)^{10}$, we similarly have
$$
\#\{n\leq x: \o(n,T_{q^j})>CU\} \ll x\left(\frac{O(\log_2 x)}{Cq^jU}\right)^{10U}
$$
where $C\geq10$ is a large absolute constant. Recalling (\ref{4.1}), putting everything together and using the union bound, we obtain (provided that $C$ is large enough)
\begin{align}\label{4.11}
\#\{n\leq x: \exists q>M \mbox{ with } v_q(n)>10CU\} \ll xM^{-9} \ll x(\log_2x)^{-3}, 
\end{align} 
and such numbers contribute $o(x)$ to the left side of (\ref{4.6}).

It remains to note that numbers $n\leq x$ which are not counted in (\ref{4.11}) obey (\ref{4.7}) and therefore contribute $\ll xU$ to the left side of (\ref{4.6}). This concludes the proof of the upper bound in Theorem \ref{th1.3}.

\subsection{Lower bound} We do a somewhat similar but more delicate analysis for the lower bound in Theorem \ref{th1.3}. 
For a subset of primes $T$, define
$$
H''(T)=\sum_{p\in T}\frac{1}{p^2}
$$
and recall the definition (\ref{4.2}) of $H(T)$. We will need the notion of the total variation distance $d_{TV}(X,Y)$ between two random variables living on the same discrete space $\O$:
\begin{equation}\label{4.12}
d_{TV}(X,Y)=\sup_{E\subset \O}\left|\P(X \in E)-\P(Y \in E)\right|.
\end{equation}

The following is a result of Ford \cite{Ford20}.

\begin{theorem}\label{th4.2}
Let $2\leq y\leq x$ and $T_1,\ldots,T_r$ be disjoint subsets of primes in $[2,y]$. Let $n\leq x$ be drawn uniformly at random from $[1,x]\cap\Z$, and, for each $i=1,\ldots,r$, let $Z_i$ be a Poisson random variable with parameter $H(T_i)$. Assume that $Z_1,\ldots,Z_r$ are independent. Then 
$$
d_{TV}\Big((\o(n,T_1),\ldots, \o(n,T_r)),(Z_1,\ldots,Z_r)\Big) \ll \sum_{i=1}^r\frac{H''(T_i)}{1+H(T_i)} +u^{-u},
$$	
where $u=\frac{\log x}{\log y}$ and the implied constant is absolute.
\end{theorem}

We will also need the standard estimate for tails of Poisson distribution. For the reader's convenience, we provide a simple proof here.

\begin{lem}\label{lem4.3}
For Poisson random variable $Y$ with parameter $\l\geq1$ and any $a\in\R$.	
\begin{equation*}\label{}
\P(Y=\l+a\l^{1/2}) \ll \l^{-1/2}e^{-0.3\min\{a^2, \, a\l^{1/2}\}},	
\end{equation*}
\end{lem}

\begin{proof} For $x>0$, denote 
$$
Q(x)=x\log x-x+1.
$$
By Stirling's formula, one has
$$
\P(Y=k)=e^{-\l}\frac{\l^k}{k!} \ll k^{-1/2}\exp\big(-\l Q(k/\l)\big),
$$
and the claim follows from the inequalities $Q(1+x) \geq x^2/3$ for $|x|\leq1$, $Q(1+x)\geq (\log4-1)x$ for $x\geq1$. 
\end{proof}

Let $c>0$ be a small absolute constant and $q_1,\ldots,q_s$ be the primes in $[cM,2cM]$ with 
\begin{equation}\label{4.13}
q_i-q_{i-1}\geq 2^{-8}\log_3x,	
\end{equation}
and note that Lemma \ref{lem3.1} guarantees that $s\gg c\pi(M)$. By (\ref{4.4}), it is enough to show that, for at least $x/4$ numbers $n\leq x$, 
\begin{equation}\label{4.14}
\#\{v_{q_i}(n): 1\leq i\leq s\} \geq s/4.	
\end{equation}
Define 
\begin{equation}\label{4.15}
y=\exp\big(\log x/\log_3x\big)
\end{equation}
and, for each $i=1,\ldots,s$, let
$$
T_{q_i}^y=\{p\leq y: p\equiv 1\pmod {q_i}\}
$$
be the truncated version of the set $T_{q_i}$.  
Our first goal will be to show that 
\begin{equation}\label{4.16}
\#\{\o(n,T_{q_i}^y): 1 \leq i\leq s\} \geq s/3	
\end{equation}
for at least $x/3$ numbers $n\leq x$. To analyze the behavior of $\o(n,T_{q_i}^y)$, we would like to apply Theorem \ref{th4.2}; however, one cannot use it directly since the sets $T_{q_i}^y$ are not disjoint. To overcome this obstacle, we introduce the following ``refinement'' of these sets. For non-empty subsets $A\subseteq [s]:=\{1,\ldots,s\}$, we consider the sets of primes
$$
S_A=\{p\leq y: p\equiv 1 \pmod{q_i} \mbox{ for } i\in A, \quad p\notequiv 1 \pmod{q_i} \mbox{ for } i\notin A\}
$$
and the mutually independent Poisson random variables
$$
Z_A=\Pois(H(S_A)).
$$
Clearly, $S_A$ and $S_B$ are disjoint for any distinct $A,B\subseteq [s]$. Hence, for the random vectors 
\begin{align*}
\bo=\big(\o(n,S_A): A\subseteq [s], |A|\geq 1\big), \quad \bZ=(Z_A: A\subseteq [s], |A|\geq 1),
\end{align*}
(in the first one we again assume that $n\leq x$ is taken uniformly at random), Theorem \ref{th4.2} gives 
\begin{equation}\label{4.17}
d_{TV}(\bo,\bZ) \ll \sum_{p>cM}\frac{1}{p^2}+(\log_3x)^{-\log_3x} \ll (cM)^{-1},
\end{equation}
which means that the vector $\bo$ has distribution close to that of $\bZ$. %In particular, for any event $E \subseteq \N_0^{\{A \subseteq [s]\}}$, we have $|\P(\o\in E) -\P(\bZ\in E)| \ll M^{-1}$.
Now we define the new random variables %(for convenience, below we write $Z_i$ instead of $Z_{\{i\}}$ and $S_i$ instead of $S_{\{i\}}$, slightly abusing notation)
$$
X_i=\sum_{A\ni i}Z_A, \quad i=1,\ldots, s,
$$
which are the ``idealized'' analogs of $\o(n,T_{q_i}^y)$. Then (\ref{4.16}) is reduced to the following purely probabilistic statement.

\begin{lem}\label{lem4.4}
With probability at least $1/2$, the random vector $\bZ$ obeys
\begin{equation}\label{4.18}
\#\{X_i: i=1,\ldots,s\} \geq s/3.
\end{equation}  
\end{lem}

\begin{proof} Let 
$$
N=\#\{X_i: i=1,\ldots,s\}
$$	
for short. We start with 
\begin{equation}\label{4.19}
\E N=\sum_{k\geq0}\P\left(\exists i: X_i=k\right).  	
\end{equation}	
Using the inequality 
$$
\P\left(\bigvee_{i=1}^sE_i\right)\geq \sum_{i=1}^s\P(E_i)-\sum_{1\leq i<j\leq s}\P(E_i\wedge E_j)
$$
(which is valid for any events $E_i$; see, for instance, \cite[Exercise 1.1.3]{TV}) for each $k$, we get
\begin{align}\label{4.20}
\E N\geq \sum_{k\geq0}\sum_{i=1}^s\P(X_i=k)-\sum_{k\geq0}\sum_{1\leq i<j\leq s}\P(X_i=X_j=k)
=s-\sum_{1\leq i<j\leq s}\P(X_i=X_j).
\end{align}
For $i<j$, the event $X_i=X_j$ is equivalent to
\begin{align}\label{4.21}
\sum_{A\ni i, A\notni j}Z_A=\sum_{A\ni j, A\notni i}Z_A.
\end{align}
The left side here, which we denote by $X_{ij}$, is the Poisson random variable with parameter
\begin{align*}
&\l_{ij}:=\sum_{A\ni i, A\notni j}H(S_A)%=\sum_{\substack{p\leq y\\ p\equiv1\pmod{q_i}\\ p\notequiv 1\pmod{q_j}}}\frac1p
=\sum_{\substack{p\leq y\\ p\equiv1\pmod{q_i}}}\frac1p-\-\sum_{\substack{p\leq y\\ p\equiv1\pmod{q_iq_j}}}\frac1p=\frac{\log_2x}{q_i}+o(1);
\end{align*}
here we used the Siegel-Walfisz theorem (e.g., \cite[Theorem 12.1]{D}). Similarly, the right side $X_{ji}$ of (\ref{4.21}) is the Poisson random variable with parameter
$$
\l_{ji}=\frac{\log_2x}{q_j}+o(1).
$$
Note also that $X_{ij}$ and $X_{ji}$ are independent. By Lemma \ref{lem4.3}, 
$$
\max_{u\geq0}\P(X_{ij}=u) \ll \l_{ij}^{-1/2} \asymp cs^{-1/2}.
$$
Trivially, there are at most $s^{3/2}$ pairs $(i,j)$ with $0<j-i\leq s^{1/2}$, and thus
\begin{equation*}
\sum_{0<j-i\leq s^{1/2}}\P(X_i=X_j) \ll cs.	
\end{equation*}
Now assume that $2^ls^{1/2}<j-i\leq 2^{l+1}s^{1/2}$ for some $l\geq0$. By (\ref{4.13}), we have $q_j-q_i \geq 2^{l-8}s^{1/2}\log_3x$ and
$$
\l_{ij}-\l_{ji}=\frac{(q_j-q_i)\log_2x}{q_iq_j}+o(1) \gg c^{-2}2^ls^{1/2} \gg c^{-1}2^l\max\{\l_{ij}^{1/2}, \l_{ji}^{1/2}\}
$$ 
provided that $x$ is large enough. Thus, (\ref{4.21}) implies either
$$
|X_{ij}-\l_{ij}|\gg c^{-1}2^l\l_{ij}^{1/2}
$$
or 
$$
|X_{ji}-\l_{ji}| \gg c^{-1}2^l\l_{ji}^{1/2}.
$$
Putting everything together and using Lemma \ref{lem4.3} again, we obtain (for some absolute constant $c_3>0$)
\begin{align*}
\sum_{1\leq i<j\leq s}\P(X_i=X_j) \ll cs+\sum_{l\geq0}2^ls^{3/2}(cs^{-1/2})e^{-c_3c^{-1}2^l} \leq 0.01s	
\end{align*} 
 provided that $c>0$ is small enough. Therefore, (\ref{4.20}) gives
 $$
 \E N \geq 0.99s.
 $$
Also, we trivially have $N\leq s$, and thus $\P(N\geq s/3) \geq 1/2$. The claim follows.
\end{proof}
 
From Lemma \ref{lem4.4}, (\ref{4.17}) and the definition (\ref{4.12}) of the total variation distance, we deduce that
$$
 \P(\bo \mbox{ obeys } (\ref{4.16}))\geq 1/2-o(1)\geq 1/3,
 $$ 
 and this is equivalent to say that (\ref{4.16}) holds for at least $x/3$ numbers $n\leq x$. Now, having established (\ref{4.16}), we need to take into account the remaining primes to go from $\o(n,T_{q_i}^y)$ to the desired quantities $v_{q_i}(n)$. For each $i=1,\ldots,s$, we have
\begin{equation}\label{4.24}
v_{q_i}(n)-\o(n,T_{q_i}^y)=\sum_{j\geq2}\o(n,T_{q_i^j})+\o(n,T_{q_i}^{y,x}),	
\end{equation} 
where
$$
T_{q_i}^{y,x}=\{y<p\leq x: p\equiv 1\pmod{q_i} \}.
$$
We will use Markov's inequality to prove that the right side of (\ref{4.24}) is not too large for most $n\leq x$. Basically, we show that, for most $n\leq x$, the vectors $(v_{q_i}(n): 1\leq i\leq s)$ and $(\o(n, T_{q_i}^y): 1\leq i\leq s)$ are close to each other in $l_1$-norm. By the definition (\ref{4.15}) of $y$ and the Brun-Titchmarsh inequality, we have
\begin{align*}
\frac1x\sum_{n\leq x}\sum_{i=1}^s\left(\sum_{j\geq2}\o(n,T_{q_i^j})+\o(n,T_{q_i}^{y,x})\right)\ll \sum_{i=1}^s\left(\sum_{j\geq2}\frac{\log_2x}{q_i^j}+\frac{\log_4x}{q_i}\right) \ll \frac{s}{\log_3x}.	
\end{align*}
It follows that for all but $o(x)$ numbers $n\leq x$ we have
\begin{equation}\label{4.25}
\sum_{i=1}^s\left|v_{q_i}(n)-\o(n,T_{q_i}^y)\right| \ll \frac{s}{(\log_3x)^{1/2}}	
\end{equation}
(say). Now for at least $x/4$ numbers $n\leq x$ which obey both (\ref{4.16}) and (\ref{4.25}) we have (\ref{4.14}), since adding a vector of $l_1$-norm at most $o(s)$ to a vector $z\in \N_0^s$ can reduce the number of distinct components of $z$ by at most $o(s)$. This completes the proof of the lower bound in Theorem \ref{th1.3}.

%\section{Appendix: maximal orders of $h(n)$ and $h(\phi(n))$}

%To show (\ref{1.1}), it is enough to consider the number $n=p_1^kp_2^{k-1}\ldots p_{k-1}^2p_k$ (here $2=p_1<p_2<\ldots$ are consecutive primes) with $k=(2-o(1))(\log x/\log_2x)^{1/2}$. 

%One can similarly establish (\ref{1.2}): if $q_1<\ldots<q_k$ are all primes in $(y/2,y]$ and $n=q_1^kq_2^{k-1}\ldots  q_k$ then $P^+((q_1-1)\ldots(q_k-1))<q_1$ and thus $h(\varphi(n))\geq k$. Since crudely $\log n\leq k\sum_{y/2<q\leq y}\log q=(1/4+o(1))y^2/\log y$, we can take $y=(\log x\log_2x)^{1/2}$ to guarantee $n\leq x$, and (\ref{1.2}) follows.

\end{document}